\documentclass[a4paper,11pt]{amsart}
 \usepackage{float}
\usepackage{amssymb}
\usepackage{mathrsfs}
\usepackage{amsmath,amssymb,amsthm,latexsym,amscd,mathrsfs}
\usepackage{indentfirst}
\usepackage{stmaryrd}
\usepackage{graphicx}
\usepackage{extarrows}
\usepackage{multirow}
\usepackage{latexsym}
\usepackage{amsfonts}
\usepackage{color}
\usepackage{pictexwd,dcpic}
\usepackage{graphicx}
\usepackage{psfrag}
\usepackage{hyperref}
\usepackage{comment}
\usepackage{enumerate}

\makeatletter
\numberwithin{equation}{section} \theoremstyle{plain}
\newtheorem{thm}{Theorem}[section]

\newtheorem{prop}[thm]{Proposition}
\newtheorem{lem}[thm]{Lemma}
\newtheorem{cor}[thm]{Corollary}
\newtheorem{defn}[thm]{Definition}

\newtheorem{exmp}[thm]{Example}

\newtheorem{rem}[thm]{Remark}

\newtheorem{ack}{Acknowledgements}

\def\<{\langle}
\def\>{\rangle}
\def\({\left(}
\def\){\right)}
\def\[{\left[}
\def\]{\right]}

\makeatother

\title{H\MakeLowercase{olomorphic d-scalar curvature on almost} H\MakeLowercase{ermitian manifolds}}

\author[J.Q. Ge]{Jianquan Ge}
\address{School of Mathematical Sciences, Laboratory of Mathematics and Complex Systems, Beijing Normal University, Beijing 100875, P.R. CHINA.}
\email{jqge@bnu.edu.cn}

\author[Y. Zhou]{Yi Zhou$^{*}$}
\email{zhou\_yi@mail.bnu.edu.cn}

\date{}
\thanks {$^{*}$ Corresponding author.}

\begin{document}
\maketitle
\begin{abstract}
In this paper, we study the existence of constant holomorphic d-scalar curvature and
the prescribing holomorphic d-scalar curvature problem on closed, connected almost Hermitian manifolds
of dimension $n\geq6$.
In addition, we obtain an application and a variational formula for the associated conformal invariant.
\end{abstract}

\section{Introduction}
Let $(M^n, g, J)$ be an almost Hermitian manifold of dimension $n$,
and let $\nabla$ be the Levi-Civita connection of $(M^n, g)$.
The Laplace operator is defined by $\Delta_g:=-\operatorname{tr}_g\nabla^2$.
The curvature of $(M^n, g)$ is defined by
$$R(X, Y)Z:=\nabla_{[X, Y]}Z-[\nabla_X, \nabla_Y]Z,$$
and the Riemannian curvature tensor is $R(X, Y, Z, W):=g(R(X, Y)Z, W)$.
Then the Ricci tensor is
$$\operatorname{Ric}(X, Y):=\operatorname{tr}_g\left\{Z\mapsto R(X, Z)Y\right\},$$
and the $\verb"*"$-Ricci tensor \cite{{Tachibana, Tang}} is
$$\operatorname{Ric}^*(X, Y):=\operatorname{tr}_g\left\{Z\mapsto -J(R(X, Z)(JY))\right\}.$$
In addition,
$R_g:=\operatorname{tr}_g\operatorname{Ric}$ and $R_g^*:=\operatorname{tr}_g\operatorname{Ric}^*$
are called the scalar curvature and the $\verb"*"$-scalar curvature, respectively.
We call the difference $$S_J:=R_g-R_g^*$$ the \emph{holomorphic d-scalar curvature} of $(M^n, g, J)$.
Hern\'{a}ndez-Lamoneda \cite{Hern} pointed out that for any smooth local orthonormal frame $\(Z_i\)_{i=1}^{n/2}$ of the holomorphic tangent bundle $T^{1, 0}M$,
$$S_J=4\sum_{i,j=1}^{n/2}(\mathcal{R}^{\mathbb{C}}(Z_i\wedge Z_j), \overline{Z_i\wedge Z_j}),$$
where
$\mathcal{R}^{\mathbb{C}}$ denotes the complex linear extension of the curvature operator $\mathcal{R}$
and $(\cdot, \cdot)$ is the complex linear extension of the inner product induced by $g$.
In \cite{Hern}, $S_J$ is named the holomorphic scalar curvature,
but we wish to avoid mistaking it for the trace of holomorphic bisectional curvature.

\begin{exmp}\label{ex-Kaehler}
For every K\"{a}hler manifold, we have $R_g=R_g^*$ and hence $S_J=0$.
In particular, every $2$-dimensional almost Hermitian manifold has $S_J=0$.
\end{exmp}

\begin{exmp}
The unit sphere $(\mathbb{S}^6, \mathring{g})$ admits an almost Hermitian structure
induced by the Cayley numbers.
In this case, we have $$R_{\mathring{g}}=30,\quad R_{\mathring{g}}^*=6\quad \mbox{and}\quad S_J=24.$$
\end{exmp}

The smooth $2$-form
$$\omega(X, Y):=g(JX, Y)$$
is called the fundamental form of $(M^n, g, J)$.
The smooth tensor field
$$N_J(X, Y):=[JX, JY]-[X, Y]-J[JX, Y]-J[X, JY]$$
is called the Nijenhuis tensor of $J$.
It is well known that $J$ is integrable if and only if $N_J=0$.
For each $x\in M$, let $$W:=\left\{\alpha\in T_x^*M\otimes T_x^*M\otimes T_x^*M:
\alpha(X, Y, Z)=-\alpha(X, Z, Y)=-\alpha(X, JY, JZ)\right\}.$$
By decomposing $$\nabla\omega\big|_x\in W=W_1\oplus W_2\oplus W_3\oplus W_4$$ into $4$ components according to symmetries,
Gray and Hervella \cite{Gray-Hervella} classified all almost Hermitian manifolds into $2^4=16$ classes
which are widely used in the study of almost complex geometry.
In the following table, we only list some of the classes that will be concerned in this paper.

\begin{table}[htb]
\caption{Partial classes of Gray-Hervella classification}
\centering
\scalebox{0.9}{
\begin{tabular}{|c|c|c|}
\hline
Class & Name & Conditions\\
\hline
$\{0\}$ & K\"{a}hler & $\nabla\omega=0$\\
\hline
$\mathcal{W}_1$ & Nearly K\"{a}hler & $\nabla_X\omega(X, Y)=0$\\
\hline
$\mathcal{W}_2$ & Almost K\"{a}hler & $d\omega=0$\\
\hline
$\mathcal{W}_3$ & Balanced & $\delta\omega=0=N_J$\\
\hline
$\mathcal{W}_4$ & No name & $\nabla_X\omega(Y, Z)=
\frac{-1}{n-2}\big[g(X, Y)\delta\omega(Z)-g(X, Z)\delta\omega(Y)$\\
\quad & \quad & \quad\quad\quad\quad\quad\quad $-g(X, JY)\delta\omega(JZ)+g(X, JZ)\delta\omega(JY)\big]$\\
\hline
$\mathcal{W}_2\oplus\mathcal{W}_3$ & No name &
$\delta\omega=0=\mathfrak{C}_{XYZ}\left\{\nabla_Z\omega(X, Y)-\nabla_{JZ}\omega(JX, Y)\right\}$\\
\hline
$\mathcal{W}_3\oplus\mathcal{W}_4$ & Hermitian & $N_J=0$\\
\hline
$\mathcal{W}_1\oplus\mathcal{W}_3\oplus\mathcal{W}_4$ & $\mathcal{G}_1$ &
$\nabla_X\omega(X, Y)-\nabla_{JX}\omega(JX, Y)=0$\\
\hline
$\mathcal{W}$ & Almost Hermitian & No condition\\
\hline
\end{tabular}}
\end{table}

Here, $\mathcal{W}_i$ denotes the corresponding class of $W_i$, and
$\mathcal{W}_i\oplus\mathcal{W}_j$ denotes the corresponding class of $W_i\oplus W_j$, etc.
In addition, $\delta$ denotes the codifferential, and $\mathfrak{C}_{XYZ}$ denotes the cyclic sum for smooth vector fields $X, Y, Z\in\mathscr{X}(M)$.

In this paper, we introduce a Yamabe-type conformal invariant $Y(M^n,g,J)$ to study the existence of constant holomorphic d-scalar curvature and
the prescribing holomorphic d-scalar curvature problem on closed, connected almost Hermitian manifolds. It turns out that the answers are completely similar to that in original Yamabe problem except possibly for dimension $n=4$ (see Theorem \ref{thm chsc} and \ref{thm phsc} in Sections \ref{sec-chsc} and \ref{sec-phsc}, respectively). In Section \ref{sec-balanced}, we obtain some rigidity results for the conformal invariant $Y(M^n,g,J)$ with respect to the Gray-Hervella classification of almost Hermitian manifolds (see Theorem \ref{thm-balanced}). In particular, we show that the almost complex structure $J$ is not integrable if $Y(M^n,g,J)<0$ for $n\geq6$. We finally give a variational formula of $Y(M^n,g,J)$ under deformations of compatible almost complex structures $J$ (see Theorem \ref{dY}). The critical point $J$ is characterized by the symmetry of the $\verb"*"$-Ricci curvature of $(M^n, g, J)$ (see Theorem \ref{cp}).

\section{Conformal deformation and the Yamabe problem}\label{sec-chsc}
\subsection{Conformal deformation}\

Let $(M^n, g)$ be a closed, connected Riemannian manifold of dimension $n>2$.
Consider the conformal metric $\widetilde{g}=u^{p-2}g$ on $M$,
where $u$ is a smooth positive function and $p=\frac{2n}{n-2}$.
If $R_g$ and $R_{\widetilde{g}}$ denote the scalar curvature of $(M^n, g)$ and $(M^n, \widetilde{g})$, respectively, then we have
\begin{equation}\label{scalar eq}
4\frac{n-1}{n-2}\Delta_gu +R_gu=R_{\widetilde{g}}u^{p-1}.
\end{equation}
Thus $\widetilde{g}$ has constant scalar curvature $c$ if and only if
$u$ satisfies the well-known Yamabe equation:
\begin{equation}\label{Yamabe eq}
4\frac{n-1}{n-2}\Delta_gu +R_gu=cu^{p-1}.
\end{equation}
For every positive function $u\in C^{\infty}(M)$,
$$Q_g(u):
=\frac{\displaystyle{\int_M}R_{\widetilde{g}}\ dV_{\widetilde{g}}}
{\(\displaystyle{\int_M} dV_{\widetilde{g}}\)^{2/p}}
=\frac{\displaystyle{\int_M} 4\frac{n-1}{n-2}|du|^2+R_gu^2\ dV_g}{\|u\|^2_p}$$
is called the Yamabe functional, where $$\|u\|_p:=\(\int_M |u|^p\ dV_g\)^{\frac1p}.$$
A positive function $u$ is a critical point of $Q_g$
if and only if $u$ satisfies (\ref{Yamabe eq}) with $c=Q_g(u)/\|u\|^{p-2}_p$.
The so-called Yamabe constant
$$Y(M, g):=\inf\left\{Q_g(u): u\in C^{\infty}(M), u>0\right\}$$
is a conformal invariant.

\begin{rem}$($See \cite[p49]{Lee-Parker}$)$.
Note that $Q_g$ is continuous on the Sobolev space $H^1(M)$ and $Q_g(u)=Q_g(|u|)$.
Since $C^{\infty}(M)$ is dense in $H^1(M)$ and
a nonnegative function can be approximated in $H^1(M)$ by positive functions,
we actually have $$Y(M, g)=\inf_{u\in H^1(M)\setminus\{0\}}Q_g(u).$$
\end{rem}

The study of the Yamabe equation yielded the famous result.
\begin{thm}$($See \cite{Aubin76, Trudinger, Yamabe, Schoen}$)$.
Let $(M^n, g)$ be a closed, connected Riemannian manifold.
Then there exists a conformal metric $\widetilde{g}$ with constant scalar curvature
$R_{\widetilde{g}}=Y(M, g)$.
\end{thm}

From now on, we suppose $(M^n, g, J)$ is a closed, connected almost Hermitian manifold,
$n\geq 6$.
It is clear that $(M^n, \widetilde{g}, J)$ is still an almost Hermitian manifold.
Suppose $R_{g}^*$ and $R_{\widetilde{g}}^*$ denote the $\verb"*"$-scalar curvature of
$(M^n, g, J)$ and $(M^n, \widetilde{g}, J)$, respectively,
del Rio and Simanca \cite{Rio-Simanca} showed that
\begin{equation}\label{*scalar eq}
\frac4{n-2}\Delta_gu +R_{g}^*u=R_{\widetilde{g}}^*u^{p-1}.
\end{equation}
They also defined a conformal invariant, and proved the following result:
\begin{thm}$($See \cite{Rio-Simanca}$)$.
Let $(M^n, g, J)$ be a closed, connected almost Hermitian manifold, $n\geq 6$.
Then there exists a conformal metric $\widetilde{g}$ such that
$(M^n, \widetilde{g}, J)$ has constant $\verb"*"$-scalar curvature $R_{\widetilde{g}}^*$.
\end{thm}

\begin{rem}
In \cite{Rio-Simanca}, they claimed that the above theorem also holds for $n=4$,
but we find a gap in their estimate.
We think another approach would be required to prove this case as in the classical Yamabe problem for dimension less than $6$.
\end{rem}

As \cite[p199]{Rio-Simanca} said, it will be rare that an almost Hermitian manifold
has both constant scalar curvature and constant $\verb"*"$-scalar curvature.
Therefore, it is nontrivial to study the Yamabe problem
for holomorphic d-scalar curvature on almost Hermitian manifolds.
In the following, we discuss this problem.

Firstly, by (\ref{scalar eq}) and (\ref{*scalar eq}), we have
\begin{equation}\label{key eq}
4\Delta_gu +S_Ju=\widetilde{S_J}u^{p-1},
\end{equation}
where $S_J$ and $\widetilde{S_J}$ denote the holomorphic d-scalar curvature of
$(M^n, g, J)$ and $(M^n, \widetilde{g}, J)$, respectively.
Naturally, we consider the functional
\begin{equation}\label{QgJ}
Q_{g, J}(u):=\frac{\displaystyle{\int_M}\widetilde{S_J} dV_{\widetilde{g}}}
{\(\displaystyle{\int_M} dV_{\widetilde{g}}\)^{\frac2p}}
=\frac{\displaystyle{\int_M} 4|du|^2+S_Ju^2dV_g}{\|u\|^2_p}.
\end{equation}
Since for any $v\in C^{\infty}(M)$, $$\frac{d}{dt}\bigg|_{t=0}Q_{g, J}(u+tv)=
\frac2{\|u\|_p^2}\int_M(4\Delta_gu+S_Ju-Q_{g, J}(u)\|u\|^{2-p}_pu^{p-1})v\ dV_g,$$
a positive function $u$ is a critical point of $Q_{g, J}$
if and only if $u$ satisfies (\ref{key eq}) with $$\widetilde{S_J}=Q_{g, J}(u)/\|u\|^{p-2}_p.$$
Therefore, we can also define analogously the following conformal invariant
\begin{equation}\label{lambdagJ}
Y(M, g, J):=\inf\left\{Q_{g, J}(u): u\in C^{\infty}(M), u>0\right\}.
\end{equation}

\subsection{Some results from PDEs}\

There are two theorems about the equation
\begin{equation}\label{general eq}
4\frac{n-1}{n-2}\Delta_gu +hu=\lambda fu^{p-1},
\end{equation}
where $h, f\in C^{\infty}(M)$ with $f>0$, and $\lambda$ is a real number to be determined.

\begin{thm}$($See \cite[p131]{Aubin82}$)$.
Let $$I(u):=\(\displaystyle{\int_M} 4\frac{n-1}{n-2}|du|^2
+hu^2\ dV_g\)\bigg{/}\(\displaystyle{\int_M}fu^p\ dV_g\)^{\frac2p}$$
and $$\nu:=\inf\limits_{u\in H^1(M)\setminus\{0\}}I(u).$$
Then $$\nu\leq n(n-1)\omega_n^{\frac2n}(\sup f)^{-\frac2p},$$
where $\omega_n$ denotes the volume of $(\mathbb{S}^n, \mathring{g})$.
In addition, if $$\nu< n(n-1)\omega_n^{\frac2n}(\sup f)^{-\frac2p},$$
then $(\ref{general eq})$ has a smooth positive solution for $\lambda=\nu$.
\end{thm}

\begin{thm}$($See \cite[p131]{Aubin82}$)$.
Let $n\geq4$, and let $P$ be a point where $f$ achieves the maximum.
If $$h(P)-R_g(P)+\frac{n-4}2\frac{\Delta f(P)}{f(P)}<0,$$
then $(\ref{general eq})$ has a smooth positive solution for $\lambda=\nu$.
\end{thm}

If we take $$h=\frac{n-1}{n-2}S_J,\ f=\frac{n-1}{n-2}$$ in $(\ref{general eq})$,
then $$\nu=\(\frac{n-1}{n-2}\)^{\frac2n}Y(M, g, J)$$ and $\Delta f=0$.
Hence we have the following results:
\begin{prop}\label{prop 1}
Let $(M^n, g, J)$ be a closed, connected almost Hermitian manifold.
Then $$Y(M, g, J)\leq n(n-2)\omega_n^{\frac2n}.$$
In addition, if $$Y(M, g, J)< n(n-2)\omega_n^{\frac2n},$$
then there exists a conformal metric $\widetilde{g}$ such that
$(M^n, \widetilde{g}, J)$ has constant holomorphic d-scalar curvature $\widetilde{S_J}=Y(M, g, J)$.
\end{prop}

\begin{prop}\label{prop 2}
Let $(M^n, g, J)$ be a closed, connected almost Hermitian manifold.
If there exists a point $P\in M$ such that
\begin{equation}\label{condition}
\frac{n-1}{n-2}S_J(P)-R_g(P)<0,
\end{equation}
then there exists a conformal metric $\widetilde{g}$ such that
$(M^n, \widetilde{g}, J)$ has constant holomorphic d-scalar curvature $\widetilde{S_J}=Y(M, g, J)$.
\end{prop}

\subsection{Existence of constant holomorphic d-scalar curvature}\

Let $W$ be the Weyl tensor of $(M^n, g)$, i.e.,
$$R=-\frac{R_g}{2(n-1)(n-2)}g\owedge g+\frac1{n-2}\operatorname{Ric}\owedge g+W,$$
where $\owedge$ denotes the Kulkarni-Nomizu product, that is,
$$\begin{aligned}
h\owedge k(X, Y, Z, U):=&h(X, Z)k(Y, U)+h(Y, U)k(X, Z)\\
&-h(X, U)k(Y, Z)-h(Y, Z)k(X, U)
\end{aligned}$$
for any symmetric 2-tensors $h,k$ and any smooth vector fields $X, Y, Z, W\in\mathscr{X}(M)$ \cite[p47]{Besse}.
(Our sign convention is different from \cite{Rio-Simanca, Lee-Parker}.)
It was proved in \cite{Rio-Simanca} the following useful formula:
\begin{equation}\label{lem R-S}
(n-1)R_g^*-R_g=2(n-1)\widehat{W}(\omega^{\#}, \omega^{\#}),
\end{equation}
where $\widehat{W}$ is defined by $\widehat{W}(X\wedge Y, Z\wedge W):=W(X, Y, Z, W)$
and $\omega^{\#}$ denotes the dual tensor of the fundamental form $\omega$.

By formula $($\ref{lem R-S}$)$, inequality (\ref{condition}) is equivalent to
$$\widehat{W}(\omega^{\#}, \omega^{\#})(P)>0.$$
Therefore, we are going to prove the existence in the following three cases:
\begin{enumerate}
\item There exists $P\in M$ such that $\widehat{W}(\omega^{\#}, \omega^{\#})(P)>0$.
\item $\widehat{W}(\omega^{\#}, \omega^{\#})\equiv0$ on $M$.
\item $\widehat{W}(\omega^{\#}, \omega^{\#})\leq0$ and there exists $P\in M$ such that $\widehat{W}(\omega^{\#}, \omega^{\#})(P)<0$.
\end{enumerate}

For the case (1), Proposition \ref{prop 2} implies the following result.
\begin{prop}\label{prop c1}
Let $(M^n, g, J)$ be a closed, connected almost Hermitian manifold.
If there exists a point $P\in M$ such that
$\widehat{W}(\omega^{\#}, \omega^{\#})(P)>0$,
then there exists a conformal metric $\widetilde{g}$ such that
$(M^n, \widetilde{g}, J)$ has constant holomorphic d-scalar curvature $\widetilde{S_J}=Y(M, g, J)$.
\end{prop}

For the case (2), formula (\ref{lem R-S}) implies $(n-1)R_g^*=R_g$.
It follows that  $$S_J=R_g-R_g^*=\frac{n-2}{n-1}R_g,$$
and hence, $$Y(M, g, J)=\frac{n-2}{n-1}Y(M, g).$$
Theorem B and Theorem C in \cite{Lee-Parker} mean that
$$Y(M, g)\leq n(n-1)\omega_n^{\frac2n},$$ and the equality holds if and only if
$(M^n, g)$ is conformally equivalent to the round sphere.
Therefore, $$Y(M, g, J)= n(n-2)\omega_n^{\frac2n}$$ if and only if
$(M^n, g)$ is conformally equivalent to $(\mathbb{S}^6, \mathring{g})$, i.e.,
there exists a diffeomorphism $\varphi: M\rightarrow \mathbb{S}^6$ such that
$g$ is conformal to $\varphi^*\mathring{g}$.
Then for $(M^n, \varphi^*\mathring{g}, J)$, we still have $\widehat{W}(\omega^{\#}, \omega^{\#})\equiv0$
and hence $$S_J=\frac{n-2}{n-1}R_g=\frac{4}{5}\cdot30=24.$$
Now by Proposition \ref{prop 1}, we can conclude the case (2) with the following result.
\begin{prop}\label{prop c2}
Let $(M^n, g, J)$ be a closed, connected almost Hermitian manifold.
Suppose $\widehat{W}(\omega^{\#}, \omega^{\#})\equiv0$ on $M$.
\begin{enumerate}
\item[$(a)$] If $Y(M, g, J)< n(n-2)\omega_n^{\frac2n}$,
then there exists a metric $\widetilde{g}$ on $M$ which is conformal to $g$
and has constant holomorphic d-scalar curvature $\widetilde{S_J}=Y(M, g, J)$.
\item[$(b)$] $Y(M, g, J)=n(n-2)\omega_n^{\frac2n}$ if and only if
$(M^n, g)$ is conformally equivalent to $(\mathbb{S}^6, \mathring{g})$. If this happens,
then there exists a conformal metric $\widetilde{g}$ such that
$(M^n, \widetilde{g}, J)$ has constant holomorphic d-scalar curvature $\widetilde{S_J}=24$.
\end{enumerate}
\end{prop}

For the case (3), we are going to prove
$$Y(M, g, J)< n(n-2)\omega_n^{\frac2n}$$ by the approach in \cite{Rio-Simanca}.
By the assumption of case (3), we can take a point $P\in M$ such that
$$\widehat{W}(\omega^{\#}, \omega^{\#})(P)<0.$$
Next, we need the following results.

\begin{lem}$($See \cite[Lemma 5.5]{Lee-Parker}$)$.\label{lem detg}
Let $(M, g)$ be a Riemannian manifold, and let $P\in M$.
In $g$-normal coordinates centered at $P$, the function $\det g_{ij}$ has the expansion
$$\det g_{ij}=1-\frac13\operatorname{Ric}_{ij}(P)x_ix_j+O(|x|^3).$$
\end{lem}

\begin{thm}$($See \cite[Theorem 5.2]{Lee-Parker}$)$.\label{thm cnc}
Let $(M, g)$ be a Riemannian manifold, and let $P\in M$.
Let $k\geq0$, and let $T$ be a symmetric $(k+2)$-tensor on $T_PM$.
Then there exists a unique homogeneous polynomial $f$ of degree $k+2$ in $g$-normal coordinates such that
the conformal metric $\widetilde{g}=e^{2f}g$ satisfies
$$\operatorname{Sym}(\widetilde{\nabla}^k\widetilde{\operatorname{Ric}}(P))=T,$$
where $\widetilde{\nabla}$ and $\widetilde{\operatorname{Ric}}$ denote the Levi-Civita connection and Ricci tensor of $(M, \widetilde{g})$ respectively.
\end{thm}

Let $l$ be a real number to be determined, and let
$$T=l\operatorname{Sym}\(-\frac12\sum\limits_{i=1}^nW(e_i, \cdot, J\cdot, Je_i)\)(P).$$
By Theorem \ref{thm cnc}, we have a homogeneous polynomial $f$ of degree $2$
in $g$-normal coordinates centered at $P$ such that the conformal metric $\widetilde{g}=e^{2f}g$ satisfies
$\widetilde{\operatorname{Ric}}(P)=T(P)$.
Since $\operatorname{tr}_gT=l\widehat{W}(\omega^{\#}, \omega^{\#})(P)$ and $f(P)=0$, we have
$$R_{\widetilde{g}}(P)=\operatorname{tr}_{\widetilde{g}}\widetilde{\operatorname{Ric}}(P)
=\operatorname{tr}_{\widetilde{g}}T=\operatorname{tr}_gT=l\widehat{W}(\omega^{\#}, \omega^{\#})(P).$$
Note that $\widehat{W}(\omega^{\#}, \omega^{\#})(P)$ is invariant under this conformal deformation.
We can assume without loss of generality that
\begin{equation}\label{s.W}
R_g(P)=l\widehat{W}(\omega^{\#}, \omega^{\#})(P),
\end{equation}
and then by formula $($\ref{lem R-S}$)$, we have
\begin{equation}\label{S_J.W}
S_J(P)=\(\frac{n-2}{n-1}l-2\)\widehat{W}(\omega^{\#}, \omega^{\#})(P).
\end{equation}
In the $g$-normal coordinates centered at $P$ (fixed above for $T$),
let $\eta$ be a radial cut-off function supported in the ball
$B_{2\varepsilon}$ with $\eta=1$ in $B_{\varepsilon}$.
We consider the test function $\varphi=\eta u_{\alpha}$, where
$$u_{\alpha}(x):=\(\frac{\alpha}{|x|^2+\alpha^2}\)^{\frac{n-2}2}$$ and $0<\alpha\ll\varepsilon\ll1$.
The function $u_\alpha$ satisfies the elliptic equation
$$\Delta u_{\alpha}+n(n-2)u_{\alpha}^{p-1}=0$$ on $\mathbb{R}^n$,
and $$4n(n-1)\(\int_{\mathbb{R}^n} |u_\alpha|^p\ dx\)^{\frac2n}
=\frac{\displaystyle{\int_{\mathbb{R}^n}} 4\frac{n-1}{n-2}|du_\alpha|^2\ dx}
{\(\displaystyle{\int_{\mathbb{R}^n}} |u_\alpha|^p\ dx\)^{\frac2p}}
=Y(\mathbb{S}^n, \mathring{g})=n(n-1)\omega_n^{\frac2n}.$$
Hence we have the estimate (see \cite[Chapter 5]{Schoen-Yau} for detailed calculations)
\begin{equation}\label{estim u}
\int_{B_{\varepsilon}}|du_{\alpha}|^2dx\leq\frac14n(n-2)\omega_n^{\frac2n}\|\varphi\|_p^2.
\end{equation}
Setting $r=|x|$, we also have the following lemma.
\begin{lem}$($See \cite[Lemma 3.5]{Lee-Parker}$)$.\label{lem u}
Suppose $k>-n$. Then as $\alpha\rightarrow0$,
$$\int_0^{\varepsilon}r^ku_{\alpha}^2r^{n-1} dr=\left\{
\begin{aligned}
O(\alpha^{k+2})\ \ \ \ \ \ \ \ \,&\mbox{if}\ n>k+4;\\
O(\alpha^{k+2}\ln\frac1{\alpha})\ \ \ &\mbox{if}\ n=k+4;\\
O(\alpha^{n-2})\ \ \ \ \ \ \ \ \,&\mbox{if}\ n<k+4.
\end{aligned}\right.$$
\end{lem}

By Lemma \ref{lem detg} and Taylor's Theorem, we have
$$\begin{aligned}
Q_{g, J}(\varphi)\|\varphi\|_p^2=&\int_M 4|d\varphi|^2+S_J\varphi^2\ dV_g\\
=&\int_{B_{\varepsilon}} 4|du_{\alpha}|^2+S_Ju_{\alpha}^2\ dV_g
+\int_{B_{2\varepsilon}\setminus B_{\varepsilon}} 4|d\varphi|^2+S_J\varphi^2\ dV_g\\
\leq&\int_{B_{\varepsilon}}4|du_{\alpha}|^2\(1-\frac16\operatorname{Ric}_{ij}(P)x_ix_j+O(r^3)\) dx\\
&+\int_{B_{\varepsilon}}S_Ju_{\alpha}^2\(1+O(r^2)\) dx+E_1\\
\leq&4\int_{B_{\varepsilon}}|du_{\alpha}|^2dx
-\frac23\int_{B_{\varepsilon}}\operatorname{Ric}_{ij}(P)x_ix_j|du_{\alpha}|^2\ dx
+\int_{B_{\varepsilon}}|du_{\alpha}|^2O(r^3)\ dx\\
&+\int_{B_{\varepsilon}}\(S_J(P)+O(r)\)u_{\alpha}^2\ dx
+\int_{B_{\varepsilon}}S_Ju_{\alpha}^2O(r^2)\ dx+E_1\\
=&:4\int_{B_{\varepsilon}}|du_{\alpha}|^2dx+E_1+E_2+E_3,\\
\end{aligned}$$ where
$$E_1=\int_{B_{2\varepsilon}\setminus B_{\varepsilon}}\(4|d\varphi|^2+S_J\varphi^2\)\(1+O(r^2)\) dx,$$
$$E_2=\int_{B_{\varepsilon}}|du_{\alpha}|^2O(r^3)\ dx
+\int_{B_{\varepsilon}}O(r)u_{\alpha}^2\ dx+\int_{B_{\varepsilon}}S_Ju_{\alpha}^2O(r^2)\ dx,$$
$$E_3=-\frac23\int_{B_{\varepsilon}}\operatorname{Ric}_{ij}(P)x_ix_j|du_{\alpha}|^2\ dx
+\int_{B_{\varepsilon}}S_J(P)u_{\alpha}^2\ dx.$$
Note that $u_{\alpha}^2(x)\leq\alpha^{n-2}r^{4-2n}$ and
$$|du_{\alpha}|^2=(n-2)^2\frac{|x|^2}{\alpha^2}\(\frac{\alpha}{|x|^2+\alpha^2}\)^n
=(n-2)^2\(\frac{r}{r^2+\alpha^2}\)^2u_{\alpha}^2.$$
Then by Lemma \ref{lem u}, it is not hard to see that
\begin{equation}\label{E1}
\begin{aligned}
E_1\leq C(n)\alpha^{n-2}\int_{\varepsilon}^{2\varepsilon}r^{1-n}\ dr=O(\alpha^{n-2})
\end{aligned}
\end{equation}
and for $n\geq6$,
\begin{equation}\label{E2}
E_2\leq C(n)\int_0^{\varepsilon}ru_{\alpha}^2r^{n-1}\ dr=O(\alpha^3).
\end{equation}
Finally, we deal with $E_3$.
By (\ref{s.W}) and (\ref{S_J.W}), we have
$$\begin{aligned}
E_3&=-\frac23\frac{\omega_{n-1}}{n}R_g(P)\int_0^\varepsilon r^{n+1}|du_{\alpha}|^2\ dr
+\omega_{n-1}S_J(P)\int_0^\varepsilon r^{n-1}u_{\alpha}^2\ dr\\
&=\omega_{n-1}\widehat{W}(\omega^{\#}, \omega^{\#})(P)\int_0^\varepsilon
\[-\frac{2l}{3n}r^{n+1}|du_{\alpha}|^2+\(\frac{n-2}{n-1}l-2\)r^{n-1}u_{\alpha}^2\] dr\\
&=\omega_{n-1}\widehat{W}(\omega^{\#}, \omega^{\#})(P)
\int_0^\varepsilon\[-\frac{2(n-2)^2r^4}{3n(r^2+\alpha^2)^2}l+\frac{n-2}{n-1}l-2\]r^{n-1}u_{\alpha}^2\ dr.
\end{aligned}$$
Setting $r=\alpha\sigma$ and $l=-\frac{3n(n-1)}{n-2}$, we have
$$\begin{aligned}
&\int_0^{\varepsilon}
\[-\frac{2(n-2)^2r^4}{3n(r^2+\alpha^2)^2}l+\frac{n-2}{n-1}l-2\]r^{n-1}u_{\alpha}^2\ dr\\
=&\alpha^2\int_0^{\frac\varepsilon\alpha}
\[-\frac{2(n-2)^2\sigma^4}{3n(\sigma^2+1)^2}l+\frac{n-2}{n-1}l-2\]
\frac{\sigma^{n-1}}{(\sigma^2+1)^{n-2}}\ d\sigma\\
=&\alpha^2\int_0^{\frac\varepsilon\alpha}
\[-\frac{2(n-2)^2}{3n}\sigma^4l+\(\frac{n-2}{n-1}l-2\)(\sigma^2+1)^2\]
\frac{\sigma^{n-1}}{(\sigma^2+1)^n}\ d\sigma\\
=&\alpha^2\int_0^{\frac\varepsilon\alpha}
\[(2n^2-9n+2)\sigma^4-2(3n+2)\sigma^2-(3n+2)\]
\frac{\sigma^{n-1}}{(\sigma^2+1)^n}\ d\sigma.
\end{aligned}$$
Basic calculations show that
$$\lim_{a\rightarrow+\infty}\int_0^a
\[(2n^2-9n+2)\sigma^4-2(3n+2)\sigma^2-(3n+2)\]\frac{\sigma^{n-1}}{(\sigma^2+1)^n}\ d\sigma
=c(n)>0,$$
where $c(n)=
\frac{(m^2-2m-1)(m-1)!^2}{(m-1)(m-2)(2m-3)!}$ with $n=2m\geq6$.
Hence for sufficiently small $\alpha$, there is a positive constant $C(n)$ such that
\begin{equation}\label{E3}
E_3=C(n)\widehat{W}(\omega^{\#}, \omega^{\#})(P)\alpha^2.
\end{equation}
Therefore, (\ref{estim u})-(\ref{E3}) imply that
$$Q_{g, J}(\varphi)\|\varphi\|_p^2\leq
n(n-2)\omega_n^{\frac2n}\|\varphi\|_p^2+C(n)\widehat{W}(\omega^{\#}, \omega^{\#})(P)\alpha^2+o(\alpha^2).$$
As $\widehat{W}(\omega^{\#}, \omega^{\#})(P)<0$, by taking $\alpha$ sufficiently small, we get
$$Y(M, g, J)\leq Q_{g, J}(\varphi)<n(n-2)\omega_n^{\frac2n}.$$
Then by Proposition \ref{prop 1}, we have proved
\begin{prop}\label{prop c3}
Let $(M^n, g, J)$ be a closed, connected almost Hermitian manifold, $n\geq6$.
If $\widehat{W}(\omega^{\#}, \omega^{\#})\leq0$ and there exists $P\in M$ such that $\widehat{W}(\omega^{\#}, \omega^{\#})(P)<0$,
then there exists a conformal metric $\widetilde{g}$ such that
$(M^n, \widetilde{g}, J)$ has constant holomorphic d-scalar curvature $\widetilde{S_J}=Y(M, g, J)$.
\end{prop}
Combining Proposition \ref{prop c1}, \ref{prop c2} and \ref{prop c3}, we have proven the following theorem.
\begin{thm}\label{thm chsc}
Let $(M^n, g, J)$ be a closed, connected almost Hermitian manifold, $n\geq 6$.
Then there exists a conformal metric with constant holomorphic d-scalar curvature
$\widetilde{S_J}=Y(M, g, J)$.
\end{thm}

\section{Conformal equivalence and prescribing curvature}\label{sec-phsc}
\subsection{Conformal equivalence for almost Hermitian manifolds}\

Let $(M^n, g, J)$ be an almost Hermitian manifold, and let $\varphi: M\rightarrow M$ be a diffeomorphism.
Since $J$ is not necessarily compatible with $\varphi^*g$, we consider another almost complex structure
$J_\varphi:=d\varphi^{-1}\circ J\circ d\varphi$.
Since $$\varphi^*g(J_\varphi X, J_\varphi Y)
=g(J\circ d\varphi(X), J\circ d\varphi(Y))=g(d\varphi(X), d\varphi(Y))=\varphi^*g(X, Y),$$
$(M^n, \varphi^*g, J_\varphi)$ is an almost Hermitian manifold.
For the sake of completeness, we verify some basic properties as follows.


\begin{lem}\label{lem CE}
Let $(M^n, g, J)$ be an almost Hermitian manifold, and let $\varphi: M\rightarrow M$ be a diffeomorphism.
\begin{enumerate}
\item[$(1)$] The $\verb"*"$-Ricci tensor of $(M^n, \varphi^*g, J_\varphi)$ is $\varphi^*\operatorname{Ric}^*$.
\item[$(2)$] The $\verb"*"$-scalar curvature of $(M^n, \varphi^*g, J_\varphi)$ is $R_g^*\circ\varphi$.
\item[$(3)$] The holomorphic d-scalar curvature of $(M^n, \varphi^*g, J_\varphi)$ is $S_J\circ\varphi$.
\end{enumerate}
\end{lem}
\begin{proof}
To prove (1), assume $\{e_i\}$ is a smooth local orthonormal frame of $(M^n, g)$.
Then $\{d\varphi^{-1}(e_i)\}$ is a smooth local orthonormal frame of $(M^n, \varphi^*g)$
and thus the $\verb"*"$-Ricci tensor of $(M^n, \varphi^*g, J_\varphi)$ is
$$\begin{aligned}
(X, Y)\mapsto&-\sum_{i=1}^{n}
\varphi^*g(J_\varphi(\varphi^*R(X, d\varphi^{-1}(e_i))(J_\varphi Y)), d\varphi^{-1}(e_i))\\
=&\sum_{i=1}^{n}\varphi^*g(\varphi^*R(X, d\varphi^{-1}(e_i))(J_\varphi Y), J_\varphi\circ d\varphi^{-1}(e_i))\\
=&\sum_{i=1}^{n}\varphi^*R(X, d\varphi^{-1}(e_i), J_\varphi Y, d\varphi^{-1}\circ J(e_i))\\
=&\sum_{i=1}^{n}R(d\varphi(X), e_i, J\circ d\varphi(Y), Je_i)\\
=&(\varphi^*\operatorname{Ric}^*)(X, Y).
\end{aligned}$$
$(2)$ follows from $(1)$ by taking trace.
Since the scalar curvature of $(M^n, \varphi^*g)$ is $R_g\circ\varphi$,
$(3)$ follows from $(2)$ immediately.
\end{proof}

Consider the metric $\widehat{g}=\varphi^*(u^{p-2}g)$ on $M$, where $p=\frac{2n}{n-2}$.
Then Lemma \ref{lem CE} and equation (\ref{key eq}) imply the following lemma.
\begin{lem}\label{S_J CE}
Suppose $S_J$ and $\widehat{S_J}$ denote the holomorphic d-scalar curvature of
$(M^n, g, J)$ and $(M^n, \widehat{g}, J_\varphi)$, respectively.
Then $$4\Delta_gu +S_Ju=\(\widehat{S_J}\circ\varphi^{-1}\)u^{p-1}.$$
\end{lem}

\subsection{Prescribing holomorphic d-scalar curvature}\

Let $(M^n, g, J)$ be a closed, connected almost Hermitian manifold with holomorphic d-scalar curvature $S_J$.
Given $K\in C^{\infty}(M)$, we want to realize $K$ as
the holomorphic d-scalar curvature of some $(M^n, \widehat{g}, J_\varphi)$.
Let $$\mathcal{S}(M, g, J):=\left\{\mbox{holomorphic d-scalar curvature functions of some}\
(M^n, \widehat{g}, J_\varphi)\right\}.$$
By the approach of Kazdan and Warner \cite{Kazdan-Warner2},
we can obtain the following results for prescribing holomorphic d-scalar curvature.
\begin{lem}\label{lem K-Z}
If there is a constant $c>0$ such that $\min\limits_M K<cS_J<\max\limits_M K$,
then $K\in\mathcal{S}(M, g, J)$.
\end{lem}
\begin{proof}
The proof of \cite[Lemma 6.1]{Kazdan-Warner2} actually deals with a general operator
$$T(u):=u^{-a}(\alpha\Delta_g u+ku),$$
where $a>1$, $\alpha>0$ and $k\in C^{\infty}(M)$
(In this paper, the sign of Laplacian operator is different from \cite{Kazdan-Warner2, Kazdan-Warner1}).
In fact, if there is a constant $c>0$ such that $\min\limits_M K<ck<\max\limits_M K$,
then there exist a positive function $u\in C^{\infty}(M)$ and a diffeomorphism $\varphi$ of $M$
such that $T(u)=K\circ\varphi^{-1}$.
By Lemma \ref{S_J CE}, we only need to take $a=p-1$, $\alpha=4$ and $k=S_J$.
\end{proof}

The first eigenvalue of the operator $L_g(u)=4\frac{n-1}{n-2}\Delta_gu +R_gu$
plays an important role in Kazdan and Warner's work \cite{Kazdan-Warner2, Kazdan-Warner1}.
Here, we need the first eigenvalue $\lambda_1(g, J)$ of the operator $L_{g, J}(u)=4\Delta_gu +S_Ju$.
By \cite[Theorem 2.11, Remark 2.12 and Theorem 3.2]{Kazdan-Warner1},
it is easy to see that the sign of $\lambda_1(g, J)$ is a conformal invariant.
Since $\lambda_1(\widetilde{g}, J)=\widetilde{S_{J}}$
for the conformal metric $\widetilde{g}$ such that $\widetilde{S_{J}}=Y(M, g, J)$,
we know the conformal invariant $Y(M, g, J)$ has the same sign with $\lambda_1(g, J)$.
So we can replace $\lambda_1(g, J)$ by $Y(M, g, J)$ in the statement of the following theorem.

\begin{thm}\label{thm phsc}
Let $(M^n, g, J)$ be a closed, connected almost Hermitian manifold, $n\geq6$.
\begin{enumerate}
\item[$(1)$] If $Y(M, g, J)<0$, then $\mathcal{S}(M, g, J)$ is precisely
the set of smooth functions that are negative somewhere on $M$.
\item[$(2)$] If $Y(M, g, J)=0$, then $\mathcal{S}(M, g, J)$ is precisely
the set of smooth functions that either change sign or are identically zero on $M$.
\item[$(3)$] If $Y(M, g, J)>0$, then $\mathcal{S}(M, g, J)$ is precisely
the set of smooth functions that are positive somewhere on $M$.
\end{enumerate}
\end{thm}
\begin{proof}
Combining Theorem \ref{thm chsc} and Lemma \ref{lem K-Z},
the argument of \cite[Theorem 6.2]{Kazdan-Warner2} is still valid.
\end{proof}

\begin{rem}
For the $\verb"*"$-scalar curvature, we can obtain similar results.
\end{rem}

\section{Relations to balanced metrics}\label{sec-balanced}
At first, we state the definition of conformally balanced metric in our terminology.
\begin{defn}$($See \cite[Definition 3.13]{Liu-Yang17}$)$.
Let $(M^n, g, J)$ be a Hermitian manifold.
The Riemannian metric $g$ is called a balanced metric if $\delta\omega=0$,
i.e., $(M^n, g, J)\in\mathcal{W}_3$.
$g$ is said to be conformally balanced if it is conformal to a balanced metric.
\end{defn}

Based on Gauduchon's work \cite{Gauduchon}, Fu and Zhou \cite{Fu-Zhou} showed the following result.
\begin{thm}$($See \cite[Theorem 2.1 and 2.3]{Fu-Zhou}$)$.\label{Int S_J}
Let $(M^n, g, J)$ be a closed almost Hermitian manifold.
\begin{enumerate}
\item[$(1)$] If $(M^n, g, J)\in\mathcal{W}_1\oplus\mathcal{W}_3\oplus\mathcal{W}_4$,
then $\int_M S_J\ dV_g\geq0$.
The equality holds if and only if $(M^n, g, J)$ is a balanced manifold.
\item[$(2)$] If $(M^n, g, J)\in\mathcal{W}_2\oplus\mathcal{W}_3$,
then $\int_M S_J\ dV_g\leq0$.
The equality holds if and only if $(M^n, g, J)$ is a balanced manifold.
\end{enumerate}
\end{thm}
Similar results were also found in \cite{Hern}.
Combining this with Theorem \ref{thm chsc},
we have the following properties for the conformal invariant $Y(M, g, J)$.

\begin{thm}\label{thm-balanced}
Let $(M^n, g, J)$ be a closed, connected almost Hermitian manifold, $n\geq6$.
\begin{enumerate}
\item[$(1)$] If $(M^n, g, J)\in\mathcal{W}_1\oplus\mathcal{W}_3\oplus\mathcal{W}_4$,
then $Y(M, g, J)\geq0$. In addition, $Y(M, g, J)=0$
if and only if $J$ is integrable and $g$ is conformally balanced.
\item[$(2)$] If $(M^n, g, J)\in\mathcal{W}_2\oplus\mathcal{W}_3$,
then $Y(M, g, J)\leq0$. In addition, $Y(M, g, J)=0$
if and only if $(M^n, g, J)$ is balanced.
\end{enumerate}
\end{thm}
\begin{proof}
To prove $(1)$, we point out that the class $\mathcal{W}_1\oplus\mathcal{W}_3\oplus\mathcal{W}_4$ is conformally invariant \cite[Theorem 4.2]{Gray-Hervella}.
Thus Theorem \ref{Int S_J} $(1)$ implies $Y(M, g, J)\geq0$.
If $Y(M, g, J)=0$, then by Theorem \ref{thm chsc}, there exists a conformal metric $\widetilde{g}$
with constant holomorphic d-scalar curvature $\widetilde{S_J}$=0.
Hence Theorem \ref{Int S_J} $(1)$ implies 
that $J$ is integrable and $g$ is conformally balanced.
Conversely, if $J$ is integrable and $g$ is conformally balanced,
then there exists a conformal metric $\widetilde{g}$
such that $(M^n, \widetilde{g}, J)$ is balanced.
Theorem \ref{Int S_J} $(1)$ implies that $\int_M \widetilde{S_J}\ dV_{\widetilde{g}}=0$.
Then we have $Y(M, g, J)\leq0$, and hence $Y(M, g, J)=0$.

Next we prove $(2)$. If $(M^n, g, J)\in\mathcal{W}_2\oplus\mathcal{W}_3$,
then $S_J\leq0$ (see the proof of \cite[Theorem 2.3]{Fu-Zhou}), and thus $Y(M, g, J)\leq0$.
Theorem \ref{Int S_J} $(2)$ implies that $(M^n, g, J)$ is balanced if $Y(M, g, J)=0$.
Conversely, if $(M^n, g, J)$ is balanced, then $(1)$ implies that $Y(M, g, J)\geq0$,
and hence $Y(M, g, J)=0$.
\end{proof}

\begin{cor}
Let $(M^n, g, J)$ be a closed, connected almost Hermitian manifold, $n\geq6$.
If $Y(M, g, J)<0$, then $J$ is not integrable.
\end{cor}
\begin{proof}
If $J$ is integrable, then $(M^n, g, J)\in \mathcal{W}_3\oplus\mathcal{W}_4$ and thus $Y(M, g, J)\geq0$, 
which leads to a contradiction.
\end{proof}
Fu et al. \cite{Fu-Li-Yau} constructed balanced metrics on some non-K\"{a}hler Calabi-Yau threefolds $(M^6, g, J)$ including $\sharp_k S^3\times S^3$ ($k\geq2$). By Theorem \ref{thm-balanced} (1) and Theorem \ref{thm chsc}, these Hermitian manifolds have $Y(M^6, g, J)=0$ and zero holomorphic d-scalar curvature under some conformal metric, the same as all K\"{a}hler manifolds do.

\section{A variation of $Y(M, g, J)$}\label{sec-variation}
For every compatible almost complex structure $J$ of $(\mathbb{S}^6, \mathring{g})$,
direct calculations imply that
$(\mathbb{S}^6, \mathring{g}, J)$ has constant holomorphic d-scalar curvature $24$ and
$Y(\mathbb{S}^6, \mathring{g}, J)= 24\omega_6^{\frac13}$.
Motivated by this phenomenon, we consider the variation of $Y(M, g, J)$ with respect to $J$.
We mainly use the technology of Wang and Zheng \cite{Wang-Zheng}:

\begin{thm}$($See \cite[Theorem 1.1]{Wang-Zheng}$)$.\label{lem der}
Let $f(t)$ be a continuous function on an interval $I\subset\mathbb{R}$.
Suppose that for any $t\in I$, there exists a $C^1$ function $F(t, s)$ of $s$
defined on a neighborhood of $t$ such that $F(t, t)=f(t)$ and $F(t, s)\geq f(s)$.
\begin{enumerate}
\item[$(1)$] If $\frac{\partial F}{\partial s}(t, t)$ is locally bounded,
then $f(t)$ is locally Lipschitz. 
\item[$(2)$] For any $t$ in the interior of $I$, if $f(t)$ is differentiable at $t$,
then $$f'(t)=\frac{\partial F}{\partial s}(t, t).$$
\end{enumerate}
\end{thm}

Let $(M^n, g, J)$ be a closed, connected almost Hermitian manifold, $n\geq6$.
By Theorem \ref{thm chsc} in Section \ref{sec-chsc}, we know that the set
$$C^Y(M, g, J):=\{u\in C^\infty(M): Q_{g, J}(u)=Y(M, g, J),\ \|u\|_p=1\}$$
is always nonempty.
Let $\mathcal{J}_g(M)$ denote the set of compatible almost complex structures of $(M, g)$,
and let $J(t)$ be a smooth one-parameter family in $\mathcal{J}_g(M)$, $t\in(-\varepsilon, \varepsilon)$.
Then for each $t$, there exists $u(t)\in C^Y(M, g, J(t))$ such that
$Q_{g, J(t)}(u(t))=Y(M, g, J(t))$.
For any $t, s\in(-\varepsilon, \varepsilon)$, let
\begin{equation}\label{fandF}
f(t):=Y(M, g, J(t))\ \mbox{and}\ F(t, s):=Q_{g, J(s)}(u(t)).
\end{equation}
Then we have $$F(t, t)=f(t)\ \mbox{and}\ F(t, s)\geq f(s).$$
It is easy to see that $F$ is $C^{\infty}$ on $s$, and the continuity of $f$ follows from the next lemma.

\begin{lem}\label{lem Y con}
The functional $$Y(M, g, \cdot): \mathcal{J}_g(M)\rightarrow\mathbb{R},\ J\mapsto Y(M, g, J)$$ is continuous, where $\mathcal{J}_g(M)$ is equipped with the $C^0$ topology.
\end{lem}
\begin{proof}
For any $u\in C^\infty(M)$ with $\|u\|_p=1$, the H\"{o}lder inequality implies that
$$\int_M u^2\ dV_g\leq\|u^2\|_{\frac{p}{2}}\|1\|_{\frac{n}{2}}=\(\operatorname{Vol}(M, g)\)^{\frac2n}.$$
Then for any $J_0, J\in\mathcal{J}_g(M)$,
$$\begin{aligned}
\left|Q_{g, J_0}(u)-Q_{g, J}(u)\right|&=\left|\int_M \(S_{J_0}-S_J\)u^2\ dV_g\right|
\leq\|S_{J_0}-S_J\|_{C^0}\int_M u^2\ dV_g\\
&\leq 2n\|J_0-J\|_{C^0, g}\|R\|_{C^0, g}\(\operatorname{Vol}(M, g)\)^{\frac2n},
\end{aligned}$$
where the last inequality follows from that for any local orthonormal frame,
$$\begin{aligned}
|S_{J_0}-S_J|&=\left|\sum_{i,j,k,l=1}^{n}\(J_i^kJ_j^l-(J_0)_i^k(J_0)_j^l\)R_{ijkl}\right|
\leq\sum_{i,j,k,l=1}^{n}\left|(J_0)_i^k(J_0)_j^l-J_i^kJ_j^l\right|\left|R_{ijkl}\right|\\
&=\sum_{i,j,k,l=1}^{n}\left|(J_0)_i^k\((J_0)_j^l-J_j^l\)+\((J_0)_i^k-J_i^k\)J_j^l\right|\left|R_{ijkl}\right|\\
&\leq\sum_{i,j,k,l=1}^{n}\(\left|(J_0)_j^l-J_j^l\right|+\left|(J_0)_i^k-J_i^k\right|\)\left|R_{ijkl}\right|\\
&=2\sum_{i,j,k,l=1}^{n}\left|(J_0)_i^k-J_i^k\right|\left|R_{ijkl}\right|
\leq 2n\left|J_0-J\right|_g\left|R\right|_g.
\end{aligned}$$
Thus for each $\varepsilon>0$, there exists $\delta=\delta(\varepsilon, n, g)>0$
such that for any $J_0, J\in\mathcal{J}_g(M)$ with $\|J_0-J\|_{C^0, g}<\delta$,
$$Q_{g, J_0}(u)-\varepsilon<Q_{g, J}(u)<Q_{g, J_0}(u)+\varepsilon.$$
Note that $Y(M, g, J)=Q_{g, J}(u)$ for some $u\in C^Y(M, g, J)$.
Hence we have $$Y(M, g, J_0)-\varepsilon< Y(M, g, J)< Y(M, g, J_0)+\varepsilon.$$
This implies the continuity of $Y(M, g, \cdot)$.
\end{proof}

To simplify the next formula,
we introduce the $J$-Ricci form \cite{Fu-Zhou, Simanca}
$$\rho^{J}(X, Y):=-\operatorname{Ric}^*(X, JY).$$
It is easy to verify that
$\rho^{J}(Y, X)=-\operatorname{Ric}^*(Y, JX)=\operatorname{Ric}^*(X, JY)=-\rho^{J}(X, Y)$,
i.e., $\rho^{J}$ is a smooth 2-form.

\begin{prop}
Let $(M^n, g, J)$ be a closed, connected almost Hermitian manifold, $n\geq6$.
Let $J(t)$ be a smooth one-parameter family in $\mathcal{J}_g(M)$ such that $J(0)=J$,
and let $u(t)\in C^Y(M, g, J(t))$, $t\in(-\varepsilon, \varepsilon)$.
Then $$\frac{\partial Y(M, g, J(t))}{\partial t}(t)\overset{a.e.}{=}-2\int_M
\<(J')^\flat(t), \rho^{J(t)}\>_g u^2(t)\ dV_g,$$
where $(J')^\flat$ is obtained from $J'$ by lowering an index.
\end{prop}
\begin{proof}
Let $f(t)$ and $F(t,s)$ be functions we defined in (\ref{fandF}).
For each $t\in(-\varepsilon, \varepsilon)$,
$$\frac{\partial F}{\partial s}(t, s)=\frac{\partial}{\partial s}Q_{g, J(s)}(u(t))
=-\int_M \frac{\partial R_g^*}{\partial s}(s)u^2(t)\ dV_g.$$
For any smooth local orthonormal frame $(e_i)_{i=1}^{n}$,
$$\begin{aligned}
\frac{\partial R_g^*}{\partial s}(s)
&=\frac{\partial}{\partial s}\sum_{i,j=1}^{n}R(e_i, e_j, J(s)e_i, J(s)e_j)\\
&=\sum_{i,j=1}^{n}R(e_i, e_j, J'(s)e_i, J(s)e_j)+R(e_i, e_j, J(s)e_i, J'(s)e_j)\\
&=2\sum_{i,j=1}^{n}R(e_i, e_j, J'(s)e_i, J(s)e_j)\\
&=2\sum_{i=1}^{n}\rho^{J(s)}(e_i, J'(s)e_i)
=2\<(J')^\flat(s), \rho^{J(s)}\>_g.
\end{aligned}$$
Since $\|u(t)\|_p=1$ for each $t$, the H\"{o}lder inequality implies that
$$\left|\frac{\partial F}{\partial s}(t, t)\right|
\leq \left\|\frac{\partial R_g^*}{\partial s}(t)\right\|_{C^0,g}\int_M u^2(t)\ dV_g
\leq \left\|\frac{\partial R_g^*}{\partial s}(t)\right\|_{C^0,g}\(\operatorname{Vol}(M, g)\)^{\frac2n}.$$
It follows that $\frac{\partial F}{\partial s}(t, t)$ is locally bounded.
Then Theorem \ref{lem der} implies the conclusion.
\end{proof}

\begin{prop}\label{compactness}
Let $(M^n, g, J)$ be a closed, connected almost Hermitian manifold
which is not conformally equivalent to $(\mathbb{S}^6, \mathring{g})$, $n\geq6$.
Let $\{J_k\}_{k\in\mathbb{N}}$ be a sequence of compatible almost complex structures of $(M, g)$
which $C^m$-converges to $J$, $m\geq 3$. Let $u_k\in C^Y(M, g, J_k)$ for each $k$.
Then there exists a subsequence $\{u_{k_i}\}_{i\in\mathbb{N}}$
which $C^{m-1}$-converges to a smooth function $u\in C^Y(M, g, J)$.
\end{prop}
\begin{proof}
Similar to the classical Yamabe problem, the estimation $Y(M, g, J)<n(n-2)\omega_n^{\frac2n}$ implies
a uniform bound $\|u\|_{W^{m, q},g}\leq C$ for each $u\in C^Y(M, g, J)$ and any $q>n$.
Since $\|\operatorname{Ric}^*(J_k)\|_{C^{m-2}, g}$ has uniform bound for sufficiently large $k$,
the argument of \cite[Proposition 2.4]{Macbeth} is still valid
(see also a detailed proof in Macbeth's Ph.D. thesis \cite{Macbeththesis}).
\end{proof}

\begin{thm}\label{dY}
Let $(M^n, g, J)$ be a closed, connected almost Hermitian manifold
which is not conformally equivalent to $(\mathbb{S}^6, \mathring{g})$, $n\geq6$.
Let $J(t)$ be a smooth one-parameter family in $\mathcal{J}_g(M)$ such that $J(0)=J$,
$t\in(-\varepsilon, \varepsilon)$.
Then there exists $u\in C^Y(M, g, J)$ such that
$$\frac{\partial Y(M, g, J(t))}{\partial t}(0)=-2\int_M
\<K^\flat, \rho^{J}\>_g u^2\ dV_g,$$
where $K=J'(0)$ and $K^\flat$ is obtained from $K$ by lowering an index.
\end{thm}
\begin{proof}
By Theorem \ref{lem der} $(2)$, we only need to check that $f(t)$ is differentiable at $0$.
For each $t\neq0$, we have
$$\frac{F(t, t)-F(t, 0)}{t}\leq\frac{f(t)-f(0)}{t}\leq\frac{F(0, t)-F(0, 0)}{t}.$$
For any $t, s\in(-\varepsilon, \varepsilon)$, by the mean value theorem,
there exists a $\beta(t, s)\in(-s, s)$ such that
$$F(t, s)-F(t, 0)=\frac{\partial F}{\partial s}(t, \beta(t, s))s.$$
Thus we have
\begin{equation}\label{control}
\frac{\partial F}{\partial s}(t, \beta(t, t))\leq\frac{f(t)-f(0)}{t}\leq
\frac{\partial F}{\partial s}(0, \beta(0, t)).
\end{equation}
It follows that
\begin{equation}\label{limsup}
\limsup_{t\rightarrow0}\frac{f(t)-f(0)}{t}
\leq\limsup_{t\rightarrow0}\frac{\partial F}{\partial s}(0, \beta(0, t))
=-2\int_M\<K^\flat, \rho^{J}\>_g u^2(0)\ dV_g.
\end{equation}
We can choose a sequence $\{t_k\}_{k\in\mathbb{N}}$ such that
\begin{equation}\label{liminf1}
\lim_{k\rightarrow\infty}\frac{f(t_k)-f(0)}{t_k}=\liminf_{t\rightarrow0}\frac{f(t)-f(0)}{t}.
\end{equation}
Note that $u(t_k)\in C^Y(M, g, J(t_k))$ and $\{J(t_k)\}_{k\in\mathbb{N}}$ $C^{\infty}$-converges to $J$.
Then by Proposition \ref{compactness}, we have a subsequence $\{u(t_{k_i})\}_{i\in\mathbb{N}}$ which $C^{\infty}$-converges to a function $u\in C^Y(M, g, J)$.
Hence, by (\ref{control}) and (\ref{liminf1}),
\begin{equation}\label{liminf}
\begin{aligned}
\liminf_{t\rightarrow0}\frac{f(t)-f(0)}{t}
&=\lim_{i\rightarrow\infty}\frac{f(t_{k_i})-f(0)}{t_{k_i}}\\
&\geq\lim_{i\rightarrow\infty}\frac{\partial F}{\partial s}(t_{k_i}, \beta(t_{k_i}, t_{k_i}))\\
&=-2\lim_{i\rightarrow\infty}
\int_M\<(J')^\flat(\beta(t_{k_i}, t_{k_i})), \rho^{J(\beta(t_{k_i}, t_{k_i}))}\>_g u^2(t_{k_i})\ dV_g\\
&=-2\int_M\<K^\flat, \rho^{J}\>_g u^2\ dV_g.
\end{aligned}
\end{equation}
Since we only assume $u(0)\in C^Y(M, g, J)$, we can choose $u(0)=u$.
Therefore, (\ref{limsup}) and (\ref{liminf}) imply that
$$\liminf_{t\rightarrow0}\frac{f(t)-f(0)}{t}\geq-2\int_M\<K^\flat, \rho^{J}\>_g u^2\ dV_g
\geq\limsup_{t\rightarrow0}\frac{f(t)-f(0)}{t}.$$
This means that $f(t)$ is differentiable at $0$.
\end{proof}

Simanca \cite{Simanca} considered the variation of $\int_M R_g^*\ dV_g$ with respect to both $g$ and $J$.
The term $\<K^\flat, \rho^{J}\>_g$ also appears in one of Simanca's formula,
and we are going to make this term more precise.
Since $J(t)$ is in $\mathcal{J}_g(M)$, we have
$$J(t)^2=-\operatorname{id}\quad  \mbox{and}\quad  g(J(t)X, Y)=-g(X, J(t)Y)$$
for any $X, Y\in\mathscr{X}(M)$.
It follows that $$KJ=-JK\quad  \mbox{and}\quad  g(KX, Y)=-g(X, KY),$$
or equivalently, $$K^\flat(JX, Y)=K^\flat(X, JY)\quad  \mbox{and}\quad  K^\flat(X, Y)=-K^\flat(Y, X).$$
Note that every smooth 2-form $A$ has a decomposition $A=A^{2,0+0,2}+A^{1,1}$
corresponding to the orthogonal decomposition $$\Omega^2(M)=\Omega^{2,0+0,2}(M)\oplus\Omega^{1,1}(M),$$
where $A^{2,0+0,2}(JX, JY)=-A^{2,0+0,2}(X, Y)$ and $A^{1,1}(JX, JY)=A^{1,1}(X, Y)$.
Hence, we have 
$$K^\flat\in\Omega^{2,0+0,2}(M)\quad \mbox{and}\quad  
\<K^\flat, \rho^{J}\>_g=\<K^\flat, (\rho^{J})^{2,0+0,2}\>_g.$$
In addition,
$$\begin{aligned}(\rho^{J})^{2,0+0,2}(X, Y)&=\frac12\(\rho^{J}(X, Y)-\rho^{J}(JX, JY)\)
=-\frac12\(\operatorname{Ric}^*(X, JY)+\operatorname{Ric}^*(JX, Y)\)\\
&=\frac12\(\operatorname{Ric}^*(JY, X)-\operatorname{Ric}^*(X, JY)\)
=(\operatorname{Ric}^*)^{skew}(JY, X),
\end{aligned}$$
where $$(\operatorname{Ric}^*)^{skew}(X, Y):=
\frac12\(\operatorname{Ric}^*(X, Y)-\operatorname{Ric}^*(Y, X)\)$$
denotes the skew-symmetric component of $\operatorname{Ric}^*$.
Thus, we can characterize the critical point of $Y(M, g, \cdot)$ by the following theorem.
\begin{thm}\label{cp}
Let $(M^n, g, J)$ be a closed, connected almost Hermitian manifold, $n\geq6$.
Then $J$ is a critical point of the functional $Y(M, g, \cdot)$
if and only if for each $\widetilde{g}\in[g]$ $($equivalently, for one $\widetilde{g}\in[g]$$)$,
the $\verb"*"$-Ricci curvature of $(M^n, \widetilde{g}, J)$ is symmetric.
\end{thm}
\begin{proof}
Let $K$ be a smooth $(1, 1)$-tensor field on $M$ such that $$KJ=-JK\ \mbox{and}\ g(KX, Y)=-g(X, KY).$$
Then one can verify that $J(t)=Je^{-tJK}$ is a smooth one-parameter family in $\mathcal{J}_g(M)$
such that $J(0)=J$ and $J'(0)=K$,
where $\(e^A\)_x:=\sum\limits_{k=0}^{\infty}\frac{\(A_x\)^k}{k!}$ for each $x\in M$.
Therefore, this theorem follows from Theorem \ref{dY} and the above discussions.
\end{proof}
For any K\"{a}hler manifold $(M^n, g, J)$ $(n\geq6)$, the $\verb"*"$-Ricci curvature is identical to Ricci curvature and thus $J$ is a critical point of the functional $Y(M, g, \cdot)$ with critical value $Y(M, g, J)=0$. It is interesting to know whether or when $J$ is actually an extremal point of $Y(M, g, \cdot)$ on certain K\"{a}hler manifolds (e.g., flat torus $\mathbb{T}^6$, which admits one component of K\"{a}hler structures and infinitely many components of non-K\"{a}hler integrable compatible almost complex structures \cite{KYZ}).
Following Tian and Zhang \cite{Tian-Zhang} and Kelleher and Tian \cite{Kelleher-Tian}, 
it is also interesting to study relative comparison theorem about the conformal invariant $Y(M, g, J)$ under Ricci flows.

\begin{ack}
This work was supported by Beijing Natural Science Foundation (Grant No. Z190003),
National Natural Science Foundation of China (Grant Nos. 12171037 and 12271040)
and the Fundamental Research Funds for the Central Universities.	
The authors thank sincerely Professor Fangyang Zheng for his helpful discussions and enlightening talks.
\end{ack}


\end{document}